\documentclass[12pt]{article}
\usepackage{amsmath, amsthm}\usepackage{enumerate}\usepackage{amssymb}\usepackage{bbold}
\usepackage{bbm}\usepackage{dsfont}\usepackage{color}\usepackage{xcolor}
\newtheorem{theorem}{Theorem}[section]
\newtheorem{proposition}[theorem]{Proposition}

\newtheorem{lemma}[theorem]{Lemma}
\newtheorem{example}[theorem]{Example}

\newtheorem{remark}[theorem]{Remark}

\newtheorem{definition}[theorem]{Definition}

\DeclareMathOperator{\rc}{\xrightarrow[]{\mathbb{r}}}
\DeclareMathOperator{\oc}{\xrightarrow[]{\mathbb{o}}}
\DeclareMathOperator{\cc}{\xrightarrow[]{\mathbb{c}}}

\DeclareMathOperator{\moc}{\xrightarrow[]{\mathbb{mo}}}
\DeclareMathOperator{\lmoc}{\xrightarrow[]{\mathbb{m}_l\mathbb{o}}}
\DeclareMathOperator{\rmoc}{\xrightarrow[]{\mathbb{m}_r\mathbb{o}}}
\DeclareMathOperator{\lmcc}{\xrightarrow[]{\mathbb{m}_l\mathbb{c}}}

\DeclareMathOperator{\rmcc}{\xrightarrow[]{\mathbb{m}_r\mathbb{c}}}

\DeclareMathOperator{\convcone}{\xrightarrow[]{\mathbb{c}_1}}
\DeclareMathOperator{\convctwo}{\xrightarrow[]{\mathbb{c}_2}}

\newcommand{\sol}{\text{\rm sol}}

\begin{document}

\title{Multiplicative order compact operators between vector lattices and $l$-algebras}
\maketitle
\author{\centering {Abdullah Ayd{\i}n$^{1}$$^{*}$, Svetlana Gorokhova$^{2}$}\\ 
\small $^1$ Department of Mathematics, Mu\c{s} Alparslan University, Mu\c{s}, 49250, Turkey, a.aydin@alparslan.edu.tr\\
\small $^2$ Southern Mathematical Institute of the Russian Academy of Sciences, Vladikavkaz, Russia, lanagor71@gmail.com\\
\small $*$ Corresponding Author}
\bigskip

{\bf{Abstract:} \rm 
In the present paper, we introduce and investigate the multiplicative order compact operators from vector lattices to $l$-algebras. A linear operator $T$ from a vector lattice $X$ to an $l$-algebra $E$ is said to be $\mathbb{omo}$-compact if every order bounded net $x_\alpha$ in $X$ possesses a subnet $x_{\alpha_\beta}$ such that $Tx_{\alpha_\beta}\moc y$ for some $y\in E$.
We also introduce and study $\mathbb{omo}$-$M$- and $\mathbb{omo}$-$L$-weakly compact operators from vector lattices to $l$-algebras.
\bigskip

\bigskip

\maketitle

\section{Introduction}
Compact operators play significant role in the operator theory and its applications. 
Various kinds of classical convergences, like order and relatively uniform convergences, 
in vector lattices are not topological \cite[Thm.2]{Gor}, 
\cite[Thm.5]{DEM1}, \cite[Thm.2.2]{DEM2}. Fortunately, even without any topology, 
several natural types of compact operators can be investigated (see, e.g. \cite{AEEM2}). 
In the present paper, we introduce and investigate 
$\mathbb{omo}$-compact operators from vector 
lattices to $l$-algebras. Throughout the paper, all vector lattices are assumed 
to be real and Archimedean, and all operators to be linear. We denote by 
letters $X$ and $Y$ vector lattices, and by $E$ and $F$ $l$-algebras.

A net $x_\alpha$ in $X$:  
\begin{enumerate}
\item[-]  {\em $\mathbb{o}$-converges} to $x\in X$ (shortly, $x_\alpha\oc x$), 
if there exists a net $y_\beta\downarrow 0$ 
such that, for any $\beta$, there exists $\alpha_\beta$ satisfying
$|x_\alpha-x|\leq y_\beta$ for all $\alpha\geq\alpha_\beta$; 
\item[-]   {\em  $\mathbb{r}$-converges} to $x\in X$ (shortly, $x_{\alpha}\rc x$) 
if, for some $u\in X_+$, there exists a sequence $\alpha_n$ of indexes such that
$|x_\alpha-x|\le\frac{1}{n} u$ for all $\alpha\ge\alpha_n$ 
(see, e.g. \cite[1.3.4, p.20]{Ku}). 
\end{enumerate}
An operator $T: X \to Y$ is called:
\begin{enumerate}
\item[-] \textit{$\mathbb{o}$-bounded}, if $T$ takes order bounded sets to order bounded ones.  
\item[-] \textit{regular}, if $T=T_1-T_2$ with $T_1,T_2\ge 0$;
\item[-] \textit{$\mathbb{o}$-continuous}, if $Tx_\alpha\oc 0$ whenever $x_\alpha\oc 0$;
\item[-] \textit{$\mathbb{r}$-continuous}, if $Tx_\alpha \rc 0$ whenever $x_\alpha\rc 0$;
\end{enumerate}
The set ${\cal L}_b(X,Y)$ of $\mathbb{o}$-bounded operators from $X$ to $Y$ is a vector space.
Every regular operator is $\mathbb{o}$-bounded. The set ${\cal L}_r(X,Y)$ of all regular operators 
from $X$ to $Y$ is an ordered vector space with respect to the order: $T\ge 0$ if $Tx\ge 0$ for all $x\in X_+$,  
we write ${\cal L}_r(X):= {\cal L}_r(X,X)$, ${\cal L}_b(X)={\cal L}_b(X,X)$ etc.
If $Y$ is Dedekind complete then ${\cal L}_b(X,Y)$ coincides with ${\cal L}_r(X,Y)$ and is a Dedekind complete 
vector lattice \cite[Thm.1.67]{AB1} containing the set ${\cal L}_n(X,Y)$ of 
all $\mathbb{o}$-continuous operators from $X$ to $Y$ as a band 
\cite[Thm.1.73]{AB1}. It is clear that each positive and hence each regular operator is 
$\mathbb{r}$-continuous.

Assume that vector lattices $X$ and $Y$ are equipped with linear convergences 
$\mathbb{c}_1$ and $\mathbb{c}_2$ respectively.
An operator $T:X\to Y$ is called 
\begin{enumerate}
\item[-] {\em $\mathbb{c}_1\mathbb{c}_2$-continuous} (cf. \cite[Def.1.4]{AEG}), 
whenever $x_\alpha\convcone 0$ in $X$ implies $Tx_\alpha\convctwo 0$ in $Y$. 
\end{enumerate}
In the case when $\mathbb{c}_1 = \mathbb{c}_2$, we say that 
$T$ is {\em $\mathbb{c}_1$-continuous}. 
The collection of all $\mathbb{c}_1\mathbb{c}_2$-continuous operators 
from $X$ to $Y$ is denoted by ${\cal L}_{\mathbb{c}_1\mathbb{c}_2}(X,Y)$, and if $\mathbb{c}_1 = \mathbb{c}_2$, we denote 
${\cal L}_{\mathbb{c}_1\mathbb{c}_2}(X,Y)$ by ${\cal L}_{\mathbb{c}_1}(X,Y)$, and 
${\cal L}_{\mathbb{c}_1}(X,X)$ by ${\cal L}_{\mathbb{c}_1}(X)$. }

A vector lattice $X$ is called: 
\begin{enumerate}
\item[-] \textit{${l}$-algebra}, if $X$ is an associative algebra such that 
$x\cdot y\in X_+$ whenever $x,y\in X_+$.  
\end{enumerate}
An $l$-algebra $E$ is called: 
\begin{enumerate}
\item[-] {\em $d$-algebra}, if $u\cdot(x\wedge y)=(u\cdot x)\wedge(u\cdot y)$ 
and $(x\wedge y)\cdot u=(x\cdot u)\wedge(y\cdot u)$ for all $x,y\in E$ and $u\in E_+$; 
\item[-] {\em $f$-algebra} if $x\wedge y=0$ implies $(u\cdot x)\wedge y=(x\cdot u)\wedge y=0$ for all $u\in E_+$; 
\item[-] {\em semiprime} whenever the only nilpotent element in $E$ is $0$; 
\item[-] {\em unital} if $E$ has a positive multiplicative unit.
\end{enumerate}
Any vector lattice $X$ is a commutative $f$-algebra with respect to the {\em trivial 
algebra multiplication} $x\ast y = 0$ for all $x,y \in X$.

Let $\mathbb{c}$ be a linear convergence on $E$ (see, \cite[Def.1.6]{AEG}). 
The algebra multiplication in $E$ is called
\begin{enumerate}
\item[-]  
{\em right $\mathbb{c}$-continuous} {\em  $($resp., left $\mathbb{c}$-continuous}$)$ 
if $x_\alpha\cc x$ implies $x_\alpha\cdot y\cc x\cdot y$ 
(resp., $y\cdot x_\alpha\cc y\cdot x$) every $y\in E$ (cf. \cite[Def.5.3]{AEG}).
\end{enumerate}
\begin{enumerate}
\item[-] The right $\mathbb{c}$-continuous algebra multiplication will be referred to as {\em $\mathbb{c}$-continuous multiplication}.
\end{enumerate}

\begin{example}\label{TT_K}{\em
Consider $T, T_k \in {\cal L}_r(\ell^\infty)$ defined as follows:
$Tx := l(x)\cdot\mathbb{1_N}$ and 
$T_k x  = x\cdot\mathbb{1}_{\{m\in \mathbb{N} : m \ge k\}}$ for all $x\in\ell^\infty$ and all $k\in\mathbb{N}$,
where $l$ is a positive extension to $\ell^\infty$  of the functional $l(x) = \lim_{n\to\infty}x_n$ on the space $c$ of all convergent real sequences. Clearly, $T_k\downarrow\ge 0$. If $T_k\ge S\ge 0$ in ${\cal L}_r(\ell^\infty)$ for all $k\in\mathbb{N}$ then,
for every $p\in\mathbb{N}$
$$
   T_ke_p\ge Se_p\ge 0 \ \ \  (\forall k\in\mathbb{N}),
$$ 
where $e_p=\mathbb{I}_{\{p\}}\in\ell^\infty$. Since $T_ke_p=0$ for all $k>p$ then $Se_p=0$ for all $p\in\mathbb{N}$.
As $\ell^\infty=\ker(l)\oplus\mathbb{R}\cdot\mathbb{1}_{\mathbb{N}}$, $S=s\cdot T$ for some $s\in\mathbb{R}_+$, and hence
$$
   T_2\mathbb{1}_{\mathbb{N}}=\mathbb{1}_{\{m\in\mathbb{N}:m\ge 2\}}\ge s\cdot T\mathbb{1}_{\mathbb{N}}=
   s\cdot\mathbb{1}_{\mathbb{N}},
$$
which implies $s=0$, and hence $S=0$. Thus, $T_k\downarrow 0$. However, the sequence $T\circ T_k = T$ does not 
$\mathbb{o}$-converge to $0$, showing that the algebra multiplication in ${\cal L}_r(\ell^\infty)$ is not left 
$\mathbb{o}$-continuous. This also shows that, in unital $l$-algebras, $\mathbb{o}$-convergence can be properly weaker than 
$\mathbb{mo}$-convergence.}
\end{example}

A net $x_\alpha$ in $E$ {\em $\mathbb{m}_r\mathbb{c}$-converges} ({\em $\mathbb{m}_l\mathbb{c}$-converges}) to $x$ whenever
$$
   |x_\alpha-x|\cdot u\cc 0 \ \ \ \ (\text{respectively} 
   \ \ \ u\cdot |x_\alpha-x|\cc 0) \ \ \ (\forall u\in E_+),
$$	
briefly $x_\alpha\rmcc x$ and $x_\alpha\lmcc x$. In commutative algebras  
$\mathbb{m}_l\mathbb{c} \equiv \mathbb{m}_r\mathbb{c}$. 
Since $\mathbb{m}_l\mathbb{c}$-convergence turns to 
$\mathbb{m}_r\mathbb{c}$-convergence and vice versus, if we replace the algebra multiplication in $E$ 
by ``~$\hat{\cdot}$~" defined as follows: $x \hat{\cdot} y := y\cdot x$, we restrict ourselves to 
$\mathbb{m}_r\mathbb{c}$-convergence, denoting it by {\em $\mathbb{mc}$-convergence} (cf. \cite{Ay1,Ay2,AEG}). 

Let $X$ be a Dedekind complete vector lattice. Then ${\cal L}_r(X)$ is an unital Dedekind complete $l$-algebra 
under the operator multiplication, containing ${\cal L}_n(X)$ as an $l$-subalgebra. The algebra multiplication 
is: right $\mathbb{mo}$-continuous in ${\cal L}_r(X)$; and is both left and right $\mathbb{mo}$-continuous 
in ${\cal L}_n(X)$ \cite[Thm.2.1]{AEG2}.

\begin{example}\label{Example5.1}{\em (cf. \cite[Ex.3.1]{AlEG})
Let $E$ be an $f$-algebra of all bounded real functions on $[0,1]$ which differ from a constant on at most countable set of $[0,1]$. 
Let $T:E \to E$ be an operator that assigns to each $f\in E$ the constant function 
$Tf$ on  $[0,1]$ such that the set $\{x\in [0,1] : f(x) \ne (Tf)(x)\}$ is at most countable. Then $T$ is a rank one 
continuous in $\|.\|_\infty$-norm positive operator. 
Consider the following net indexed by finite subsets  of $[0,1]$: 
$$ 
f_\alpha(x) = \left\{
\begin{array}{ccc}
1 &\text{ if } & x \not\in \alpha\\
0 &\text{ if } & x \in \alpha
\end{array}
\right.
$$
Then $f_\alpha \downarrow 0$ in $E$, yet $\|f_\alpha \|_\infty =1$ for all 
$\alpha$. Thus, $T$ is neither  $\mathbb{omo}$- nor $\mathbb{mo}$-co\-n\-ti\-nuous.
However, $T$ is $\mathbb{r}$-co\-n\-ti\-nuous and since $E$ is unital,  $T$ is $\mathbb{mr}$-co\-n\-ti\-nuous.
}
\end{example}

The structure of the paper is as follows. In Section 2, we introduce $\mathbb{omc}$-compact operators from a vector lattice to an $l$-algebra and investigate their general properties with an emphasis on $\mathbb{omo}$- and $\mathbb{omr}$-cases. In Section 3, we investigate 
the domination problem for $\mathbb{omc}$-compact operators; we define and study $\mathbb{omo}$-$M$- and $\mathbb{omo}$-$L$-weakly compact operators. For further unexplained  terminology and notations, we refer to \cite{AB1,AB2,AEG2,AEG,AGG,AEEM1,AEEM2,Hu,Ku,Pag,Vu,Za}. 

\section{The properties of $\mathbb{omc}$-compact operators}

We begin with the following definition (cf. \cite[Def.2.12]{AEG2}).

\begin{definition}{\em
A subset $A$ of an $l$-algebra $E$ is called {\em $\mathbb{m}_r\mathbb{o}$-bounded} $($resp., {\em $\mathbb{m}_l\mathbb{o}$-bounded}$)$ if the set $A\cdot u$ $($resp., $u \cdot A$$)$ is order bounded for every $u\in E_+$.
An operator $T$ from a vector lattice $X$ to an $l$-algebra $E$ is called 
{\em $\mathbb{m}_r\mathbb{o}$-bounded} $($resp., {\em$\mathbb{m}_l\mathbb{o}$-bounded}$)$ if $T$ maps order bounded subsets of $X$ into 
$\mathbb{m}_r\mathbb{o}$-bounded $($resp., $\mathbb{m}_l\mathbb{o}$-bounded$)$ subsets of $E$.}
\end{definition}

As usual, we restrict to $\mathbb{m}_r\mathbb{o}$-bounded subsets and operators, and refer to them as $\mathbb{mo}$-bounded.
In any $l$-algebra $E$ with trivial multiplication, $x\ast y = 0$ for all $x,y\in E$, each subset $A$ of $E$ is $\mathbb{mo}$-bounded
and as result, every operator from a vector lattice $X$ to such an $l$-algebra $E$ is $\mathbb{mo}$-bounded. 
For elementary properties of $\mathbb{mo}$-bounded operators in $l$-algebras, we refer the reader to the paper \cite{AEG2}.

\begin{example}\label{mo-bounded but not o-bounded}{\em (cf.
\cite[Ex.6]{AEG}). Take a free ultrafilter $\cal{U}$ on natural numbers $\mathbb{N}$. 
Then a sequence $\lambda_n$ of reals  {\em converges along} $\cal{U}$ to $\lambda$ whenever 
$\{k\in\mathbb{N}:|\lambda_k-\lambda|\le\varepsilon\}\in\cal{U}$ for every $\varepsilon>0$. 
Hence, for any element $x:=(x_n)_{n=1}^\infty \in \ell^\infty$, the sequence $x_n$ converges along $\cal{U}$ to $x_{\cal{U}}:=\lim_{\cal{U}}x_n$. 
In that case, one can define an $l$-algebra multiplication $*$ in $\ell^\infty$ by $x*y:=(\lim_{\cal{U}}x_n)\cdot(\lim_{\cal{U}}y_n)\cdot\mathbb{1}$, 
where $\mathbb{1}$ is a sequence of reals identically equal to $1$. It is easy to see that $(\ell^\infty,*)$ is a $d$-algebra.
Then the set $A = \{ke_k : k\in {\mathbb N}\}$ is ${\mathbb {mo}}$-bou\-n\-ded yet not ${\mathbb o}$-bou\-nded.
}
\end{example}

\begin{remark}\label{basic properties of mo bounded}
Let $T$ be an operator from a vector lattice $X$ to an $l$-algebra $E$. Then 
{\em
\begin{enumerate}[(i)]
\item If $T$ is $\mathbb{o}$-bounded (in particular if $T$ is regular) then $T$ is $\mathbb{m}_l\mathbb{o}$- and $\mathbb{m}_r\mathbb{o}$-bounded.
		
\item If $T$ is $\mathbb{m}_l\mathbb{o}$- or $\mathbb{m}_r\mathbb{o}$-bounded operator and $E$ is unital $l$-algebra then $T$ is order bounded.
			
\item By \cite[Thm.2.6]{AEG2}, every $\mathbb{r}$-continuous operator $T$ from an Archimedean vector lattice to an Archimedean $l$-algebra is $\mathbb{rmo}$-continuous and then, by \cite[Thm.2.15]{AEG2}, $T$ is $\mathbb{mo}$-bounded.
			
\item It follows from \cite[Lem.1.4]{AB1} that every order continuous operator is order bounded and hence $\mathbb{m}_l\mathbb{o}$- and $\mathbb{m}_r\mathbb{o}$-bounded.
			
\item Every $\mathbb{mo}$-, $\mathbb{omo}$-, or $\mathbb{rmo}$-continuous is $\mathbb{m}_l\mathbb{o}$-bounded and $\mathbb{m}_r\mathbb{o}$-bounded. Moreover, every $\mathbb{m}_l\mathbb{o}$-, $\mathbb{om}_l\mathbb{o}$-, or $\mathbb{rm}_l\mathbb{o}$-continuous 
$($resp., $\mathbb{m}_r\mathbb{o}$-, $\mathbb{om}_r\mathbb{o}$-, or $\mathbb{rm}_r\mathbb{o}$--continuous$)$ is $\mathbb{m}_l\mathbb{o}$-bounded $($resp., $\mathbb{m}_r\mathbb{o}$-bounded$)$ \cite[Thm.2.14]{AEG2}.
\end{enumerate}}
\end{remark}

The converse of Remark \ref{basic properties of mo bounded}~(i) need not to be true in general. 
Indeed, in any $l$-algebra with trivial multiplication, every operator is $\mathbb{m}_l\mathbb{o}$- and $\mathbb{m}_r\mathbb{o}$-bounded. A more interesting example is given below.

\begin{example}{\em
Consider an operator $T$ from the vector lattice $X:=c$ the set of all convergent reel sequences to the $f$-algebra $E:=c_0$ of real sequence converging to zero, defined by 
$$
T(x_1,x_2,x_3,\cdots)=(x,x-x_1,x-x_2,x-x_3,\cdots),
$$
where $x=\lim\limits_{n\to\infty}x_n$. Then $T$ is an $\mathbb{m}_l\mathbb{o}$- and $\mathbb{m}_r\mathbb{o}$-bounded operator. However, it follows from $T(0,\cdots,0,1,1,\cdots)=(1,\cdots,1,0,0\cdots)$ that $T([0,\mathbb{1}])$ is not order bounded in $E$, and so, $T$ is not order bounded.}
\end{example}

The converse of Remark \ref{basic properties of mo bounded}~(iv) need not to be true in general.
To see this, we include the following example.

\begin{example}\label{order bounded not m cont}{\em (cf. \cite[Ex.2.8]{AEG2}).
Let $(\ell^\infty,*)$ be as in Example~\ref{mo-bounded but not o-bounded}.
Now, the identity operator $I:(\ell^{\infty},*)\to(\ell^{\infty},*)$ is order bounded, but not $\mathbb{omo}$-continuous. 
Indeed, take the characteristic functions $h_n=\mathbb{1}_{\{k\in\mathbb{N}:k\ge n\}}\in\ell^\infty$.
Then $h_n\oc 0$ in $\ell^\infty$ yet the sequence $|I(h_n)-I(0)|*\mathbb{1}=h_n*\mathbb{1}=\mathbb{1}$ does not $\mathbb{o}$-converge to 0. 
Thus, the sequence $I(h_n)$ does not $\mathbb{mo}$-converge to 0 and hence $I$ is not $\mathbb{omo}$-continuous.}
\end{example}

Remind that an operator between normed spaces is called \textit{compact} if it maps the closed unit ball to a relatively compact set. Equivalently, the operator is compact if, for each norm bounded sequence, there exists a subsequence such that the image of it is convergent. Motivated by this, we introduce the following notions.

\begin{definition}{\em 
An operator $T$ from a vector lattice $X$ to an $l$-algebra $E$ is called
\begin{enumerate}[(a)]
	
\item {\em $\mathbb{om}_r\mathbb{o}$-compact} $($resp., {\em $\mathbb{om}_l\mathbb{o}$-compact}$)$ if every order bounded net $x_\alpha$ in $X$ possesses a subnet $x_{\alpha_\beta}$ such that $Tx_{\alpha_\beta}\rmoc y$ $($resp., $Tx_{\alpha_\beta}\lmoc y$$)$ for some $y\in E$;

\item {\em $\mathbb{om}\mathbb{o}$-compact} if $T$ is both $\mathbb{om}_r\mathbb{o}$- and $\mathbb{om}_l\mathbb{o}$-compact;

\item {\em sequentially $\mathbb{om}_r\mathbb{o}$-compact} $($resp., {\em $\mathbb{om}_l\mathbb{o}$-compact}) if every order bounded sequence $x_n$ in $X$ possesses a subsequence $x_{n_k}$ such that $Tx_{n_k}\rmoc y$ $($resp., $Tx_{n_k}\lmoc y$$)$ for some $y\in E$;

\item {\em sequentially $\mathbb{om}\mathbb{o}$-compact} if $T$ is both  sequentially $\mathbb{om}_r\mathbb{o}$- and $\mathbb{om}_l\mathbb{o}$-compact.
\end{enumerate}}
\end{definition}

\begin{example}\label{ExamX}
{\em
Define an operator $T:c_0 \to c_0$ by
$$
T\left( \sum_{k=1}^\infty a_k e_k\right)  = \sum_{k=1}^\infty \frac{a_k}{k} e_k, 
$$
where $e_k=\mathbb{1}_{\{n\}}$ and  $a_k$ is a real sequence converging to zero.
Then $T$ is compact on the $f$-algebra $(c_0, \|.\|_\infty)$, and is $\mathbb{omo}$-compact.
}
\end{example}

In the next example, we show that there is an operator that is neither $\mathbb{omo}$-compact nor sequentially $\mathbb{omo}$-compact.

\begin{example}\label{neither omo-compact nor sequentially omo-compact}{\em
The identity operator on the $l$-algebra $L_\infty[0,1]$ with pointwise multiplication is neither 
$\mathbb{omo}$-compact nor sequentially $\mathbb{omo}$-compact. Indeed, take the sequence of Rademacher function $r_n(t)=sgn(\sin(2^n\pi t))$ on $[0,1]$. Clearly, $r_n$ is order bounded. Now, assume that $r_n$ has a $\mathbb{mo}$-convergent subnet $r_{n_\alpha}$, say
$r_{n_\alpha}\moc f$ for some $f\in L_\infty[0,1]$. Then $r_n\oc f$ and hence $r_n(t)\to f(t)$ almost everywhere
violating that $r_n(t)$ diverges on $[0,1]$ except countably many points of form $\frac{k}{m}$ for $k,m\in {\mathbb N}$. }
\end{example}

An $\mathbb{omo}$-compact operator need not be sequentially $\mathbb{omo}$-compact. To see this, we consider 
\cite[Ex.7]{AA} for the next example.

\begin{example}{\em
Consider the set $E:=\mathbb{R}^X$ of all real-valued functions on $X$ equipped with the product topology, where $X$ is the set of all strictly increasing maps from $\mathbb{N}$ to $\mathbb{N}$. It follows from \cite[Ex.3.10(i)]{Pag} that $E$ is a unital Dedekind complete $f$-algebra with respect to the pointwise operations and ordering.
\begin{enumerate}[(i)]
\item The identity map $\mathcal{I}$ on $E$ is an $\mathbb{om}_r\mathbb{o}$-compact operator. Indeed, assume that $f_\alpha$ is an order bounded net in $E$.
It follows from \cite[Ex.7(1)]{AA} that there exists a subnet $f_{\alpha_\beta}$ such that $f_{\alpha_\beta}\oc f$ for some $f\in E$. 
Since every $f$-algebra has $\mathbb{o}$-continuity algebra multiplication, it follows from \cite[Lm.5.5]{AEG} that $f_{\alpha_\beta}\rmoc f$. Therefore, $\mathcal{I}$ is $\mathbb{om}_r\mathbb{o}$-compact.
\item The identity map $\mathcal{I}$  on $E$ is not sequentially $\mathbb{om}_r\mathbb{o}$-compact.
Consider a sequence $f_n$ in $\{-1,1\}^X$. Then $f_n$ is order bounded
in $E$ and $f_n$ has no $\mathbb{o}$-convergent subsequence \cite[Ex.7(2)]{AA}. 
Thus, every subsequence does not $\mathbb{mo}$-converge because the $f$-algebra $E$ has a unit element.
\end{enumerate}}
\end{example}

\begin{remark}{\em 
It is known that any compact operator is norm continuous, but in general we may have a $\mathbb{omo}$-compact operator which is not $\mathbb{mo}$-continuous.
Indeed, denote by $\mathcal{B}$ the Boolean algebra of the Borel subsets of $[0,1]$ equals up to measure null sets. 
Let $\mathcal{U}$ be any ultrafilter on $\mathcal{B}$. Then it can be shown that the linear operator $\varphi_{\mathcal{U}}:L_\infty[0,1]\to\mathbb{R}$ defined by 
$$ 
\varphi_{\mathcal{U}}(f):=\lim\limits_{A\in\mathcal{U}}\frac{1}{\mu(A)}\int_{A}fd\mu
$$ 
is $\mathbb{omo}$-compact (see \cite[Lem.5.5]{AEG}) because the algebra multiplication in $\mathbb{R}$ is order continuous (cf. \cite{Hu,Pag}). However, it is not $\mathbb{mo}$-continuous.
}
\end{remark}

In any $l$-algebra $E$, $x\ge y$ implies $x\cdot u\ge y\cdot u$ for all $u\in E_+$. 
But, in general, the inequality $x\cdot u\ge 0$ for all $u\in E_+$ does not imply $x\ge 0$.

\begin{definition}{\em
An $l$-algebra $E$ is called {\em right straight $l$-algebra} $($resp., {\em left straight $l$-algebra}$)$ if $x\in E_+$ 
whenever $x\cdot u\in E_+$ $($resp., $x\cdot u\in E_+$$)$ for all $u\in E_+$. If an $l$-algebra $E$ is both left and 
right straight $l$-algebra, we say that $E$ is a {\em straight $l$-algebra}.}
\end{definition}

Clearly every unital $l$-algebra is  straight. An algebra in \cite[Ex.2.8)]{AEG2} gives an example of a $d$-algebra which is not a straight $l$-algebra. The following theorem is an $\mathbb{omo}$-version of \cite[Thm.2]{AA}.
\begin{theorem}\label{compact implies order bounded}
Let $T$ be an operator from a vector lattice $X$ to an $l$-algebra $E$. Then
\begin{enumerate}[$($i$)$]
\item if $E$ is right straight $l$-algebra $($resp., left straight $l$-algebra$)$ and $T$ is $\mathbb{om}_r\mathbb{o}$-compact $($resp., $\mathbb{om}_l\mathbb{o}$-compact$)$ operator then it is order bounded;
\item if $T$ is $\mathbb{om}_r\mathbb{o}$-compact $($resp., $\mathbb{om}_l\mathbb{o}$-compact$)$ operator then it is $\mathbb{m}_r\mathbb{o}$-bounded $($resp., $\mathbb{m}_l\mathbb{o}$-bounded$)$ operator.
\end{enumerate}
\end{theorem}
	
\begin{proof}
$(i)$ Suppose that $T$ is an $\mathbb{om}_r\mathbb{o}$-compact operator, but not order bounded. 
So, there is an order bounded subset $B$ of $X$ such that $T(B)$ is not order bounded in $E$. 		
Hence, for every $y\in E_+$, there exists some $x_y\in B$ such that $|T(x_y)|\nleq y$. 		

Since the net $(x_y)_{y\in E_+}$ is order bounded, there exists a subnet $(y_v)_{v\in \mathcal{V}\subseteq E_+}=(x_{\phi(v)})_{v\in \mathcal{V}\subseteq E_+}$ of $(x_y)_{y\in E_+}$ such that $T(y_v)$ is $\mathbb{m}_r\mathbb{o}$-converges to some $z\in E$, i.e., for each positive element $w\in E_+$, $|T(y_v)-z|\cdot w\oc 0$ because $T$ is $\mathbb{om}_r\mathbb{o}$-compact operator.
So, $|T(y_v)-z|\cdot w$ has an order bounded tail, which means that for an arbitrary positive element $w\in E_+$ there exist 
some indexes $v_0\in \mathcal{V}$ and elements $e\in E_+$ such that 
$$
|T(y_v)-z|\cdot w\leq e
$$
for each $v\geq v_0$. 
It follows from the inequality $|T(y_v)|\leq |T(y_v)-z|+|z|$ that we have $|T(y_v)|\cdot w\leq |T(y_v)-z|\cdot w+|z|\cdot w\leq e+|z|\cdot w$ for every $v\geq v_0$.
Now, fix $t:=e+|z|\cdot w\in E_+$. Then we have 
\begin{eqnarray}
	|T(x_{\phi(v)})|\cdot w\leq t       
\end{eqnarray}
for all $v\geq v_0$.
Now, take an index $v_1\in \mathcal{V}\subset E_+$ so that $\phi(v)\cdot w\geq t$ holds for all $v\geq v_1$.
Then, for any $v\geq v_0\vee v_1$, we have $|T(x_{\phi(v)})|\nleq \phi(v)$, and so, $|T(x_{\phi(v)})|\cdot w\nleq \phi(v)\cdot w$ because $E$ is right straight $l$-algebra. Therefore, we have 
$$
|T(x_{\phi(v)})|\cdot w\nleq t,
$$
which is a contradiction with $(1)$. Therefore, we obtain the desired result.	

$(ii)$ The proof is a modification of the proof $(i)$.
\end{proof}

The idea in Theorem \ref{compact implies order bounded} need not to be true in the case of sequentially $\mathbb{omo}$-compactness. 
To see this, we consider \cite[Lm.4 and Ex.6]{AA} for the following example.

\begin{example}\label{fi5}{\em
Let $F$ be the $l$-algebra of all bounded real-valued functions defined
on the real line with countable support,
and $E$ be the directed sum $\mathbb{R}\mathbb{1}\oplus F$, where $\mathbb{1}$ denotes the constant function taking the value $1$.
Define an operator $T$ from $E$ to $F$ as a projection such that the range is $F$ and the kernel is $\mathbb{R}\mathbb{1}$.
Then $T$ is a sequentially $\mathbb{om}_r\mathbb{o}$-compact, but not order bounded. 
Indeed, take an order bounded sequence $(f_n)$ in $E$. Then there exist some scalars $\lambda>0$ such that $|f_n|\leq \lambda$ for all $n$. 
On the other hand, we can write $f_n=\beta_n+g_n$ with reel numbers $\beta_n$ and functions $g_n$ in $F$. 
It follows from \cite[Ex.6]{AA} that $(g_{n_k})$ is order convergent in $F$. 
Then it is clear from \cite[Thm.VIII.2.3]{Vu} that $(g_{n_k})$ is also $\mathbb{m}_r\mathbb{o}$-convergent in $F$. 
On the other hand, it is clear that the image of the net $(\mathbb{1}_{x})_{x\in[0,1]}$ is not order bounded in $F$. Therefore, the operator $T$ is not order bounded. }
\end{example}

\begin{proposition}\label{leftandrightmultiplication}
Let $E$ be an $l$-algebra and $R,T,S$ are operators on $E$.
\begin{enumerate}[{\em (i)}]
\item If $T$ is $($se\-que\-n\-tial\-ly$)$ $\mathbb{om}_r\mathbb{o}$-compact $($resp., $\mathbb{om}_l\mathbb{o}$-compact$)$ and $S$ is $($se\-que\-n\-tial\-ly$)$ $\mathbb{m}_r\mathbb{o}$-continuous $($resp., $\mathbb{m}_l\mathbb{o}$-continuous$)$ then the operator $S\circ T$ is $($se\-que\-n\-tial\-ly$)$ $\mathbb{om}_r\mathbb{o}$-compact $($resp., $\mathbb{om}_l\mathbb{o}$-compact$)$.

\item If $T$ is $($sequentially$)$ $\mathbb{om}_r\mathbb{o}$-compact $($resp., $\mathbb{om}_l\mathbb{o}$-compact$)$ and $R$ is order bounded, then $T\circ R$ is $($se\-que\-n\-tial\-ly$)$ $\mathbb{om}_r\mathbb{o}$-compact $($resp., $\mathbb{om}_l\mathbb{o}$-compact$)$.

\item If $T$ is a positive, order continuous and $\mathbb{om}_r\mathbb{o}$-compact $($resp., $\mathbb{om}_l\mathbb{o}$-compact$)$ operator, $F$ is right straight $l$-algebra 
$($resp., left straight $l$-algebra$)$, and $S_\alpha\downarrow 0$ is a decreasing net of order bounded operators 
then $T\circ S_\alpha\downarrow 0$ is a decreasing net of $\mathbb{om}_r\mathbb{o}$-compact $($resp., $\mathbb{om}_l\mathbb{o}$-compact$)$ operators.
\end{enumerate}
\end{proposition}

\begin{proof}
(i) Let $x_\alpha$ be an order bounded net in $E$. Since $T$ is $\mathbb{om}_r\mathbb{o}$-compact, there are a subnet $x_{\alpha_{\beta}}$ and $x\in E$ such that $Tx_{\alpha_{\beta}}\rmoc x$. 
It follows from the $\mathbb{m}_r\mathbb{o}$-continuity of $S$ that $S(Tx_{\alpha_{\beta}})\rmoc S(x)$. Therefore, $S\circ T$ is $\mathbb{om}_r\mathbb{o}$-compact. 

(ii) Assume $x_\alpha$ to be an order bounded net in $E$. Since $R$ is order bounded, the net $Rx_\alpha$ is order bounded. Now, the $\mathbb{om}_r\mathbb{o}$-compactness of $T$ implies 
that there are a subnet $x_{\alpha_{\beta}}$ and $z\in E$ such that $TRx_{\alpha_{\beta}}\rmoc z$. Therefore, $T\circ R$ is $\mathbb{om}_r\mathbb{o}$-compact.

(iii) Let $S_\alpha\downarrow 0$ be a net of order bounded operators. Then it follows from Theorem \ref{compact implies order bounded} 
that $T\circ S_\alpha$ is an order bounded operator, and also, $\mathbb{om}_r\mathbb{o}$-compact operator for each index $\alpha$ by $(ii)$. 
Moreover, since $T\ge 0$, we have $T\circ S_\alpha\downarrow$. On the other hand, by \cite[Thm.VIII.2.3]{Vu}, 
$T\circ S_\alpha\downarrow 0$ if and only if $T\circ S_\alpha x\downarrow 0$ for each $x\in E_+$. 
The result follows from $S_\alpha\downarrow 0$.
	
The sequential and $\mathbb{om}_l\mathbb{o}$-compact cases are analogous.
\end{proof}

\begin{proposition}\label{prop1}
Every order continuous finite rank operator on an  $l$-algebra $E$ with $\mathbb{o}$-continuous multiplication is $\mathbb{omo}$-compact. 
\end{proposition}

\begin{proof}
Let $T:E\to E$ be order continuous and $\dim(TE) <\infty$. Then 
$$
T = \sum_{k=1}^m x_k \otimes f_k 
\  \text{for} \   x_1, \dots, x_m \in E
\ \text{and} \   f_1, \dots, f_m \in E'_n.
$$
Without lost of generality, we may assume that $T = x_1 \otimes f_1$.
Since $E'_n$ is Dedekind complete, $f_1$ is regular, and $T$ is also regular. Without lost of generality, suppose  $x_1 \ge 0$ and $f_1 \ge 0$. Let $z_\alpha$ be an order bounded net in $E$.  
Then $Tz_\alpha = (x_1 \otimes f_1) (z_\alpha) =f_1 (z_\alpha)x_1$ is order bounded since every order continuous functional is order bounded. Since $\dim(TE) =1$, there exists a subnet $z_{\alpha_\beta}$ such that $Tz_{\alpha_\beta} \oc y \in T(E)$.
Using  $\dim(TE) =1$ again, we obtain $Tz_{\alpha_\beta}\moc y$. Therefore $T$ is $\mathbb{omo}$-compact. 
\end{proof}

The following result is an extension of Example~\ref{ExamX}.
\begin{proposition}\label{propX}
Let $E$ be an $l$-algebra with $\mathbb{o}$-continuous algebra multiplication. 
Then the algebra ${\cal L}_{rc}(E)$ of regular order compact operators is a subspace  
of ${\cal L}_{romo}(E)$, which is itself an right algebraic ideal of ${\cal L}_r(E)$.
\end{proposition}

\begin{proof}
Suppose that $T$ is a regular order compact operator on a right $\mathbb{o}$-continuous $l$-algebra $E$, and $x_\alpha$ is an order bounded net in $E$. Then there exist a subnet $x_{\alpha_\beta}$ and some $y\in E$ such that $Tx_{\alpha_\beta}\oc y$. It follows from \cite[Lm.5.5]{AEG} that $Tx_{\alpha_\beta}\rmoc y$. Thus, we obtain that $T$ is $\mathbb{om}_r\mathbb{o}$-compact. As the proof of $\mathbb{om}_l\mathbb{o}$-compactness is analogous, ${\cal L}_{rc}(E)$ is  subspace of ${\cal L}_{romo}(E)$. 
On the other hand, it is well known that ${\cal L}_r(E)$ is a subspace of ${\cal L}_b(E)$. 
It follows from Theorem \ref{leftandrightmultiplication}~(ii) that ${\cal L}_{romo}(E)$ is an right algebraic ideal of ${\cal L}_r(E)$.
\end{proof}

\section{Domination problem for compact operators}
	
In this section, we study the domination problem for $\mathbb{omo}$-compact operators, and we introduce the $\mathbb{omo}$-$M$- and $\mathbb{omo}$-$L$-weakly compact operators. Now, we consider the domination problem for positive  $\mathbb{mo}$($\mathbb{omo}$)-continuous and $\mathbb{omo}$-compact operators. We have a positive answer for $\mathbb{mo}$($\mathbb{omo}$)-continuous operators in the next lemma.

\begin{lemma}\label{dominated implies continuity}
Let $T$ and $S$ be positive operators between $l$-algebras $E$ and $F$ satisfying $0\le S\le T$. If $T$ is an $\mathbb{m}_r\mathbb{o}$-continuous $($resp., $\mathbb{m}_l\mathbb{o}$-, $\mathbb{om}_r\mathbb{o}$- and $\mathbb{om}_l\mathbb{o}$-continuous$)$ operator imply then $S$ has the same property.
\end{lemma}

\begin{proof}
Suppose that $T$ is an $\mathbb{m}_r\mathbb{o}$-continuous operator and $x_\alpha$ is $\mathbb{m}_r\mathbb{o}$-convergent to $x\in E$. Then we have $Tx_\alpha\rmoc Tx$ in $F$. It follows from \cite[Lem.1.6]{AB2} that
$$
0\leq |Sx_\alpha-Sx|\leq S(|x_\alpha-x|)\leq T(|x_\alpha-x|)
$$
holds for all $\alpha$ because $S$ is a positive operator. Hence, we get 
\begin{eqnarray}
|Sx_\alpha-Sx|\cdot u\leq T(|x_\alpha-x|)\cdot u
\end{eqnarray}
for all $u\in F_+$.
On the other hand, it follows from \cite[Prop.2.4]{Ay1} that $x_\alpha\rmoc x$ implies $|x_\alpha-x|\rmoc 0$, and so, we obtain $T(|x_\alpha-x|)\rmoc0$ by the $\mathbb{m}_r\mathbb{o}$-continuity of $T$, i.e., $T(|x_\alpha-x|)\cdot u\oc 0$ for all $u\in F_+$.
Hence, the desired result raises from the inequality $(2)$, $Sx_\alpha\rmoc Sx$ in $F$. The proof for the cases of $\mathbb{m}_l\mathbb{o}$-, $\mathbb{om}_r\mathbb{o}$- and $\mathbb{om}_l\mathbb{o}$-continuity are similar.
\end{proof}

Recall that a net $(x_\alpha)_{\alpha\in A}$ in an $l$-algebra is called $\mathbb{mo}$-Cauchy if the net $(x_\alpha-x_{\alpha'})_{(\alpha,\alpha')\in A\times A}$ is $\mathbb{mo}$-convergent to $0$. Moreover, an $l$-algebra is called $\mathbb{mo}$-complete if every $\mathbb{mo}$-Cauchy net is $\mathbb{mo}$-convergent; see \cite[Def.2.11]{Ay1}.

\begin{theorem}
Let $X$ be a vector lattice and $E$ be a Dedekind and sequentially $\mathbb{mo}$-complete $l$-algebra with $\mathbb{o}$-continuous algebra multiplication. 
If $T_m:X\to E$ is a sequence of sequential $\mathbb{omo}$-compact operators and $T_m\oc T$ in $\mathcal{L}_b(X,E)$ then $T$ is sequentially $\mathbb{omo}$-compact. 
\end{theorem}

\begin{proof}
Let $x_n$ be a order bounded sequence in $X$, $T_m$ be a sequence of sequential $\mathbb{om}_r\mathbb{o}$-compact operators and $E$ be  sequentially $\mathbb{m}_r\mathbb{o}$-complete. Then there is $w\in X_+$ such that $|x_n|\leq w$ for all $n\in N$.
Also, by a standard diagonal argument, there exists a subsequence $x_{n_k}$ such that for any $m\in N$, $T_mx_{n_k}\rmoc y_m$ for some $y_m\in E$. 
Let's show that $y_m$ is a $\mathbb{mo}$-Cauchy sequence in $E$. Fix an arbitrary $u\in E_+$. Then we have
$$
|y_m-y_j|\cdot u\leq |y_m-T_mx_{n_k}|\cdot u+|T_mx_{n_k}-T_jx_{n_k}|\cdot u+|T_jx_{n_k}-y_j|\cdot u.
$$
Then the first and third terms in the last inequality both order converge to zero as $m\to\infty$ and $j\to\infty$, respectively. 
Since $T_m\oc T$ in vector lattice $\mathcal{L}_b(X,E)$, we have $|T_m- T_j|\oc 0$, and so, it follows from \cite[Thm.VIII.2.3]{Vu} that $|T_m- T_j|(x)\oc 0$ for all $x\in X$. Then, by using \cite[Thm.1.67(a)]{AB1}, we obtain the inequality 
$$
|T_mx_{n_k}-T_jx_{n_k}|\cdot u\leq |T_m- T_j|(|x_{n_k}|)\cdot u \leq |T_m- T_j|(w)\cdot u.
$$
Since $E$ has $\mathbb{o}$-continuous algebra multiplication, it follows from $\cite[Lem.5.5]{AEG}$ that $|T_m- T_j|(x)\oc 0$ implies $|T_m- T_j|(w)\cdot u\oc 0$. Hence, we obtain that $|T_mx_{n_k}-T_jx_{n_k}|\cdot u\oc0$. Therefore, $y_m$ is an $\mathbb{mo}$-Cauchy. 
Now, by sequentially $\mathbb{m}_r\mathbb{o}$-completeness of $E$, there is $y\in E$ such that $y_m\rmoc y$ in $E$ as $m\to\infty$. Hence,
\begin{eqnarray*}
|Tx_{n_k}-y|\cdot u&\leq& |Tx_{n_k}-T_mx_{n_k}|\cdot u+|T_mx_{n_k}-y_m|\cdot u+|y_m-y|\cdot u\\ &\leq& |T_m-T|(|x_{n_k}|)\cdot u+|T_mx_{n_k}-y_m|\cdot u+|y_m-y|\cdot u\\ &\leq& |T_m-T|(w)\cdot u+|T_mx_{n_k}-y_m|\cdot u+|y_m-y|\cdot u.
\end{eqnarray*}
Now, for fixed $m\in N$, and as $k\to\infty$, we have
$$
\limsup\limits_{k\to\infty}|Tx_{n_k}-y|\cdot u\leq |T_m-T|(w)\cdot u+|y_m-y|\cdot u.
$$ 
But $m\in N$ is arbitrary, so $\limsup\limits_{k\to\infty}|Tx_{n_k}-y|\cdot u=0$. Thus, $|Tx_{n_k}-y|\cdot u\oc 0$, i.e., $Tx_{n_k}\rmoc y$. Therefore, $T$ is sequentially $\mathbb{om}_r\mathbb{o}$-compact.

The sequentially $\mathbb{om}_l\mathbb{o}$-compact case is analogous.
\end{proof}

In the rest of the section, we discuss $\mathbb{omo}$-$M$- and $\mathbb{omo}$-$L$-weakly compact operators.
Remind that a norm bounded operator $T$ from a normed lattice $X$ into a normed space $Y$ is called $M$-weakly compact if $Tx_n\xrightarrow{\Vert\cdot\Vert}$0 holds for every norm bounded disjoint sequence $x_n$ in $X$. 
Also, a norm bounded operator $T$ from a normed space $Y$ into a normed lattice $X$ is called $L$-weakly compact whenever $\lim\|x_n\|=0$ holds for every disjoint sequence $x_n$ in the solid hull $\sol(T(B_Y)):=\{x\in X:\exists y\in T(B_Y) \ \text{with} \ |x|\leq|y|\}$
of $T(B_Y)$, where $B_Y$ is the closed unit ball of $Y$. Similarly we have the following notions.

\begin{definition} {\em
Let $T:X\to E$ be a sequentially $\mathbb{mo}$-continuous operator.
\begin{enumerate} 
\item[(1)] If $Tx_n\moc 0$ for every order bounded disjoint sequence $x_n$ in $X$ then $T$ is said to be {\em $\mathbb{omo}$-$M$-weakly compact}.
\item[(2)] If $y_n\rmoc 0$ for every disjoint sequence $y_n$ in $\sol(T(A))$, where $A$ is any order bounded subset of $X$, then $T$ is said to be {\em $\mathbb{omo}$-$L$-weakly compact}.
\end{enumerate}}
\end{definition}

\begin{proposition}
Let $T$ be an order bounded $\sigma$-order continuous operator from a normed lattice $X$ to an $l$-algebra $E$ with $\mathbb{o}$-continuous algebra multiplication. Then $T$ is $\mathbb{omo}$-$M$- and $\mathbb{omo}$-$L$-weakly compact.
\end{proposition}

\begin{proof}
Clearly, $T$ is sequentially $\mathbb{mo}$-continuous operator because $E$ has $\mathbb{o}$-continuous algebra multiplication; see \cite[Lem.5.5]{AEG2}.
Let $x_n$ be an order bounded disjoint sequence in $X$. Then it follows from \cite[Rem.10]{AEEM1} that we get $x_n\oc 0$. Thus, we have $Tx_n\moc 0$. Therefore, $T$ is $\mathbb{omo}$-$M$-weakly compact.

Now, we show that $T$ is $\mathbb{omo}$-$L$-weakly compact. Let $A$ be an order bounded set in $X$.
Thus, $T(A)$ is order bounded, and so, $\sol(T(A))$ is an order bounded set in $E$. 
Take an arbitrary disjoint sequence $y_n$ in $\sol(T(A))$. Then, using \cite[Rem.10]{AEEM1}, we have $y_n\oc 0$, and so, $y_n\moc o$ since $E$ has $\mathbb{o}$-continuous algebra multiplication; see \cite[Lem.5.5]{AEG2}. Thus, $T$ is $\mathbb{omo}$-$L$-weakly compact.
\end{proof}

Similarly to \cite[Cor.2.3]{AlEG}, we obtain the following result.
\begin{proposition}\label{dominated}
Let $T,S:X\to E$ be two linear operators from a normed lattice $X$ to an $l$-algebra $E$ such that $0\le S\le T$. If $T$ is $\mathbb{omo}$-$M$- or $\mathbb{omo}$-$L$-weakly compact then $S$ has the same property.
\end{proposition}

\begin{proof}
Suppose that $T$ is an $\mathbb{omo}$-$M$-weakly compact operator. Thus, it follows from Lemma \ref{dominated implies continuity} that $S$ is an $\mathbb{mo}$-continuous operator. Let $x_\alpha$ be an order bounded disjoint net in $X$. So, $|x_n|$ is also order bounded and disjoint. Since $T$ is $\mathbb{omo}$-$M$-weakly compact, $T(|x_n|)\moc 0$ in $E$. 
Following from the inequality 
\begin{eqnarray}
0\leq |Sx_n|\cdot u\leq S(\lvert x_n\rvert)\cdot u\leq T(\lvert x_n\rvert)\cdot u
\end{eqnarray}
for all $n\in \mathbb{N}$ and for every $u\in E_+$ (cf. \cite[Lem.1.6]{AB2}), we get $Sx_n\moc 0$ in $E$. 
Thus, $S$ is $\mathbb{omo}$-$M$-weakly compact.

Next, we show that $S$ is $\mathbb{omo}$-$L$-weakly compact. Let $A$ be an order bounded subset of $X$. 
Put $\lvert A\rvert=\{\lvert a\rvert: a\in A\}$. Clearly, $\sol(S(A))\subseteq \sol(S(\lvert A\rvert))$ and since $0\leq S\leq T$, 
we have $\sol(S(\lvert A\rvert))\subseteq \sol(T(\lvert A\rvert))$. 
Let $y_n$ be a disjoint sequence in $\sol(S(A))$ then $y_n$ is in $\sol(T(\lvert A\rvert))$ and, since $T$ is $\mathbb{omo}$-$L$-weakly compact then $T(\lvert x_n\rvert)\moc 0$ in $E$. Therefore, by inequality $(3)$, $S$ is $\mathbb{omo}$-$L$-weakly compact.
\end{proof}

\begin{proposition}
If $T:X\to E$ is an $\mathbb{omo}$-$L$-weakly compact lattice homomorphism then $T$ is $\mathbb{omo}$-$M$-weakly compact.
\end{proposition}

\begin{proof}
Take an order bounded disjoint sequence $x_n$ in $X$. 
Since $T$ is lattice homomorphism, we have that $Tx_n$ is disjoint in $E$. 
Clearly $Tx_n\in \sol\big(\{Tx_n:n\in \mathbb{N}\}\big)$. 
By $\mathbb{omo}$-$L$-weakly compactness of $T$, we have $Tx_n\moc 0$ in $E$. Therefore, $T$ is $\mathbb{omo}$-$M$-weakly compact.
\end{proof}

{\tiny 

}
\end{document}